\documentclass[12pt]{amsart}
\usepackage{amsmath}
\usepackage{amsthm}
\usepackage{amssymb}
\usepackage{enumerate}
\usepackage{xcolor}
\usepackage{comment}

\theoremstyle{definition}

\newtheorem{Theorem}{Theorem}[section]
\newtheorem{Lemma}[Theorem]{Lemma}
\newtheorem{Corollary}[Theorem]{Corollary}
\newtheorem{Proposition}[Theorem]{Proposition}
\newtheorem{Remark}[Theorem]{Remark}

\newtheorem{Definition}[Theorem]{Definition}

\newcommand{\cal}{\mathcal}

\newcommand{\vol}{\text{Vol}}

\newcommand{\sat}{\text{sat}}

\title{Some formulas for epsilon multiplicity in local rings}

\begin{document}

\author{Stephen Landsittel }

\address{Stephen Landsittel, Department of Mathematics,
University of Missouri, Columbia, MO 65211, USA}
\email{sdlg6f@missouri.edu}

\begin{abstract}
    We prove that the epsilon multiplicity exists in a Noetherian local ring whenever the nildradical of the completion of $R$ has nonmaximal dimension. We also extend the volume equals multiplicity formula for the epsilon multiplicity to this setting.
\end{abstract}

\maketitle

\section{Introduction}

In this paper we extend a volume formula on epsilon multiplicity of ideals to local rings $R$ whose completion has a nildradical of nonmaximal dimension; that is $\dim N(\widehat{R})<\dim R$ where $N(\widehat{R})$ is the nildradical of the completion $\widehat{R}$ of $R$. This condition holds if $\widehat{R}$ is reduced at the minimal primes of maximal dimension.
\\

This is the most general condition in which we can expect limits of this type to exist, as in Theorem 1.1 \cite{C2} it is shown that the multiplicity of all graded families of $m_R$ primary ideals in a local ring $R$ exists if and only if the nildradical of $\widehat{R}$ has nonmaximal dimension.\\

If $R$ is a $d$-dimensional Noetherian local ring with completion $\widehat{R}$. Then the dimension of the nildradical of $\widehat{R}$ is less than $d$ if $\widehat{R}$ is generically reduced. The converse holds if $\widehat{R}$ is equidimensional. The definition of generically reduced is as follows.

\begin{Definition}
    Suppose that $R$ is a Noetherian local ring of dimension $d$. we say that $R$ is generically reduced if $R_P$ is a domain for all minimal primes $P$ of $R$ with maximal dimension.
\end{Definition}

Now we define the epsilon multiplicity.

\begin{Definition}
    Let $R$ be a Noetherian local ring of dimension $d$ with maximal ideal $m_R$, and let $I$ be an ideal of $R$. In \cite{UV} Ulrich and Validashti define the epsilon multiplicity of $I$ is as
    \begin{equation*}
        \varepsilon(I)
        = \limsup_{n\to\infty}\frac{\ell_R((I^n)^{\sat}/I^n)}{n^d/d!}.
    \end{equation*}The saturation $J^{\sat}$ of an ideal $J$ of $R$ is the ideal
    \begin{equation*}
        J^{\sat} := J:m_R^{\infty}
        = \cup_{i\geq 0}J:m_R^i.
    \end{equation*}If $I$ is $m_R$-primary, then $\varepsilon(I)$ exists as a limit and equals the classical multiplicity of $I$. Recent papers where epsilon multiplicities are discussed include \cite{CS}, \cite{SuD}, \cite{DDRV}, \cite{DRT}, and \cite{CL}. An example of Roberts in \cite{Ro1}, building on an example of Nagata from \cite{Na1} shows that the Rees algebra of saturated powers $\oplus_{i\geq 0}(I^i)^{\sat}$ of an ideal $I$ need not be a finitely generated $R$-algebra, so traditional methods of commutative algebra are generally not applicable. In \cite{CHST} it is shown that the epsilon multiplicity of an ideal can be irrational.
\end{Definition}

\begin{Definition}\label{dfintro2}
    If $(R,m_R)$ is a local ring, then we denote the $m_R$-adic completion of $R$ by $\widehat{R}$. A local ring $R$ is called analytically unramified if $\widehat{R}$ is reduced. 
\end{Definition}

We prove the following result, giving a very general condition for existence of epsilon multiplicity as a limit.

\begin{Theorem}\label{TheoremA} (Theorem \ref{thm3})
    Let $R$ be any $d$-dimensional Noetherian local ring such that the nildradical of $\widehat{R}$ has dimension less than $d$. Let $I\subset R$ an ideal. Then the epsilon multiplicity of $I$
    \begin{equation}\label{eq0.1}
        \varepsilon(I) = \lim_{n\to\infty}\frac{\ell_R((I^n)^{\sat}/I^n)}{n^d}
    \end{equation}exists as a limit.
\end{Theorem}
Theorem \ref{TheoremA} is proven in \cite{C4} for more general families of modules when $R$ is a local domain essentially of finite type over a perfect field $k$ such that $R/m_R$ is algebraic over $k$, and in the case that $R$ is analytically unramified in \cite{C2} Corollary 6.3. \\

Theorem \ref{TheoremA} is a consequence of Theorem \ref{thm1} and Corollary \ref{thm2}. We show that the limit of (\ref{eq0.1}) is a sum of two existent limits. Let $N$ be the nildradical of $\widehat{R}$. In Section 2, we prove Theorem \ref{TheoremA} by first reducing to the case where $R$ is complete, and we then compare the limit $\varepsilon(I(R/N))$, which exists by Corollary 6.3 \cite{C2}, to the limit of (\ref{eq0.1}) using several exact sequences. We analyze the kernels and cokernels of these exact sequences; one of these cokernels is the module $(I^n+N)^{\sat}/((I^n)^{\sat}+N)$ (where $n$ is a fixed positive integer). To analyze the module $(I^n+N)^{\sat}/((I^n)^{\sat}+N)$, we must pay careful attention to the relationship between extending ideals of $R$ to $R/N$ and saturation of ideals. In general, for ideals $I$ of $R$, $(I^{\sat})(R/N)$ need not equal $(I(R/N))^{\sat}$. They are equal if and only if $I^{\sat}+N$ equals $(I+N)^{\sat}$. In general, only $I^{\sat}+N\subset (I+N)^{\sat}$. We subvert this problem by analyzing the limit
\begin{equation*}
    \lim_{n\to\infty}\frac{\ell_R((I^n+N)^{\sat}/((I^n)^{\sat}+N))}{n^d}.
\end{equation*}We prove in Proposition \ref{prop1} that this limit exists.
\\

We will apply Swanson's Theorem 3.4 \cite{S1} in the proof of Lemma \ref{lem2}, where we show that certain families of ideals which are closely related to that of (\ref{eq0.1}) satisfy the hypothesis of Theorem 6.1 \cite{C2}.\\

Before discussing the volume equals multiplicity formula for the epsilon multiplicity, we first give a brief history of volume equals multiplicity formulas associated to graded families of ideals. A family of ideals $\cal{I} = \{I_n\}$ indexed by the natural numbers is called graded if $I_0 = R$ and $I_mI_n\subset I_{m+n}$ for all $m$ and $n$. Cutkosky showed in \cite{C2} that the volume associated to the graded family $\cal{I}$
\begin{equation*}
    \vol(\cal{I}):= \lim_{n\to\infty}\frac{\ell_R(R/I_n)}{n^d/d!}
\end{equation*}exists whenever the dimension of the nildradical of $\widehat{R}$ is less than $\dim(R)$. The volume equals multiplicity formula for the family $\cal{I}$ states that
\begin{equation*}
        \lim_{n\to\infty}\frac{\ell_R(R/I_n)}{n^d/d!}
        =\lim_{p\to\infty}\frac{e(I_p)}{p^d}
    \end{equation*}
    and was proven when $R$ is the local ring of a point on an algebraic variety over an algebraically closed field in \cite{LM}. The formula was proven when $R$ is analytically unramified in Theorem 6.5 \cite{C2} and when the dimension of the nildradical of $\widehat{R}$ is less than $\dim(R)$ in \cite{C5}, which is the most general possible case that this theorem can hold by Theorem 1.1 \cite{C2}. In this paper we extend the result of \cite{C5} to the epsilon multiplicity of an ideal $I$.
    \\

We prove a volume equals multiplicity formula for the epsilon multiplicity in a local ring whose completion is reduced at each minimal prime of maximal dimension. We show that the epsilon multiplicity is a limit of Amao multiplicities in this setting. We state the definition of Amao multiplicity in Section 2, which was defined and developed in \cite{A1}.

\begin{Theorem}\label{TheoremB} (Theorem \ref{thm5})
    Let $R$ be a $d$-dimensional Noetherian local ring and let $I\subset R$ be an ideal in $R$. Suppose that the dimension of the nildradical of $\widehat{R}$ is less than $d$. Then the limit
    \begin{equation*}
        \lim_{m\to\infty}\frac{a(I^m,(I^m)^{\sat})}{m^d}
    \end{equation*} exists, and
\begin{equation*}
    \varepsilon(I)=\lim_{m\to\infty}\frac{a(I^m,(I^m)^{\sat})}{m^d}.
\end{equation*}
\end{Theorem}

Theorem \ref{TheoremB} is proven in the case that $R$ is analytically unramified in Theorem 1.1 \cite{CL}. In order to prove Theorem \ref{TheoremB} we prove two formulas in Lemma \ref{lem5} and Proposition \ref{prop2}, and the sum of these two equations is the volume equals multiplicity formula of Theorem \ref{TheoremB}. Lemma \ref{lem5} and Proposition \ref{prop2} are the volume formulas corresponding to the limits studied in the proof of Theorem \ref{TheoremA}.

\section{Proofs of the main theorems}

We now begin proving Theorem \ref{TheoremA} regarding the epsilon multiplicity. The epsilon multiplicity was defined in Section 1. We first prove a remark.

\begin{Remark}\label{rmk2.1}
    Let $R$ be any ring and let $I,J,G\subset R$ be ideals. We have the following equations.
    \begin{enumerate}
        \item[(i)]If $R$ is local and Noetherian
        \begin{equation*}
            \bigg(\bigg(\frac{J+G}{G}\bigg)^n\bigg)^{\sat}
            = \frac{(J^n+G)^{\sat}}{G}.
        \end{equation*} 
        \item[(ii)]
        \begin{equation*}
            \bigg(\frac{J+G}{G}\bigg)^n
            = \frac{J^n+G}{G}.
        \end{equation*}
        \item[(iii)] If $G\subset I\cap J$, then 
        \begin{equation*}
            \frac{I\cap J}{G} = (I/G)\cap (J/G).
        \end{equation*}
    \end{enumerate}
\end{Remark}

\begin{proof}
    We prove (i). Let $m$ be the maximal ideal of $R$. Let $t\in \mathbb{Z}_{>0}$ be so large that
    \begin{equation*}
        \bigg(\bigg(\frac{J+G}{G}\bigg)^n\bigg)^{\sat}
        = \bigg(\frac{J+G}{G}\bigg)^n:\bigg(\frac{m+G}{G}\bigg)^t
        = \frac{J^n+G}{G}:\frac{m^t+G}{G}
    \end{equation*}We see that
    \begin{equation*}
        \begin{split}
            \frac{J^n+G}{G}:\frac{m^t+G}{G}
            &= \{x+G\in R/G\mid
            xm^t\subset J^n+G\}\\
            &=
            \{x+G\in R/G\mid
            x\in (J^n+G):_Rm^t\}
            \\&= [(J^n+G):_Rm^t]/G
            =(J^n+G)^{\sat}/G.
        \end{split}
    \end{equation*}This finishes the proof of (i).
\end{proof}

We will prove the following.

\begin{Theorem}\label{thm1}
    Let $R$ be a complete Noetherian local ring of dimension $d>0$ such that the nildradical of $R$ has dimension less than $d$. Let $I\subset R$ be an ideal. Then the limit
     \begin{equation*}
        \lim_{n\to\infty}\frac{\ell_R((I^n)^{\sat}/I^n)}{n^d}
    \end{equation*}exists.\\
\end{Theorem}

\begin{Lemma}\label{lem2-1}
    Let $R$ be any ring, and $I,J,K$ ideals such that $I\supset J$ or $I\supset K$. Then
    \begin{equation*}
        I\cap (J+K)
         = (I\cap J)+ (I\cap K).
    \end{equation*}
\end{Lemma}
\begin{proof}
Without loss of generality $I\supset J$, so that
\begin{equation}\label{eq2-1}
    (I\cap J)+ (I\cap K) = J + (I\cap K).
\end{equation} Since $I\cap J,I\cap K\subset I \cap (J+K)$, we have
\begin{equation*}
    (I\cap J)+ (I\cap K)\subset I\cap (J+K).
\end{equation*}Let $x\in I\cap (J+K)$.

Write $x = a+b$ with $a\in J$ and $b\in K$. We have $b = x-a\in I+J = I$ since $I\supset J$. Then $b\in I\cap K$ so that $x\in J+ (I\cap K)$. We now have that $I\cap (J+K)\subset J+(I\cap K) = (I\cap J)+ (I\cap K)$ by (\ref{eq2-1}). This finishes the proof of the lemma.
\end{proof}

We now fix a complete Noetherian local ring $R$ of dimension $d>0$, let $N$ be the nildradical of $R$, and assume that $\dim_R N<d$. Let $I$ be an ideal of $R$.

\begin{Lemma}\label{lem1}
    The limit
    \begin{equation*}
        \lim_{n\to\infty}\frac{\ell_R([(I^n)^{\sat}\cap N]/[I^n\cap N])}{n^d}
    \end{equation*}exists and equals zero
\end{Lemma}
\begin{proof}By Theorem 3.4 \cite{S1}, there exists $b\in\mathbb{Z}_{>0}$ such that $(I^n)^{\sat}\cap m_R^{nb} = I^n\cap m_R^{nb}$ for all $n$.
    For each $n$, we have the natural surjection
    \begin{equation*}
        \frac{(I^n)^{\sat}\cap N}{I^n\cap Nm_R^{nb}}\to 
        \frac{((I^n)^{\sat}\cap N)/(I^n\cap Nm_R^{nb})}{(I^n\cap N)/(I^n\cap Nm_R^{nb})} \cong
        \frac{(I^n)^{\sat}\cap N}{I^n\cap N}
    \end{equation*}and the natural map
    \begin{equation*}
         \frac{(I^n)^{\sat}\cap N}{I^n\cap Nm_R^{nb}}\to
         \frac{N}{Nm_R^{nb}}
    \end{equation*} which is injective since $(I^n)^{\sat}\cap m_R^{nc} = I^n\cap m_R^{nc}$. Thus
    \begin{equation*}
        \ell_R\bigg(\frac{(I^n)^{\sat}\cap N}{I^n\cap N}\bigg)
        \leq \ell_R\bigg(\frac{(I^n)^{\sat}\cap N}{I^n\cap Nm_R^{nb}}\bigg)
        \leq \ell_R\bigg(\frac{N}{Nm_R^{nb}}\bigg)
    \end{equation*}for all $n$, which completes the proof of the lemma, since $\dim_R(N)<d$.
\end{proof}

We have an exact sequence
\begin{equation*}
    0\to \frac{(I^n+N)\cap [(I^n)^{\sat}]}{I^n}\to \frac{(I^n)^{\sat}}{I^n}\to \frac{(I^n)^{\sat}+N}{I^n+N}\to 0
\end{equation*}while using Lemma \ref{lem2-1}
\begin{equation*}
    \frac{(I^n+N)\cap [(I^n)^{\sat}]}{I^n}
    = \frac{I^n + ((I^n)^{\sat}\cap N)}{I^n}
    \cong\frac{(I^n)^{\sat}\cap N}{((I^n)^{\sat}\cap N)\cap I^n}
    =\frac{(I^n)^{\sat}\cap N}{I^n\cap N}.
\end{equation*}Hence
\begin{equation*}
    0\leq \frac{\ell_R((I^n)^{\sat}/I^n)}{n^d}-\frac{\ell_R(((I^n)^{\sat}+N)/(I^n+N))}{n^d} =\frac{\ell_R([(I^n)^{\sat}\cap N]/[I^n\cap N])}{n^d}.
\end{equation*}This combined with Lemma \ref{lem1} yields the following statement.

\begin{Remark}\label{rmk2}
    $\lim_{n\to\infty}\frac{\ell_R((I^n)^{\sat}/I^n)}{n^d}$ exists if and only if $\lim_{n\to\infty}\frac{\ell_R(((I^n)^{\sat}+N)/(I^n+N))}{n^d}$ exists, in which case these two limits are equal.
\end{Remark}

Our next goal is to show that $\lim_{n\to\infty}\frac{\ell_R(((I^n)^{\sat}+N)/(I^n+N))}{n^d}$ exists. By Corollary 6.3 \cite{C2}, we have that

\begin{equation}\label{eq1}
    \begin{split}
        \frac{1}{d!}\varepsilon(I(R/N))
    &= \frac{1}{d!}\varepsilon((I+N)/N)
    =\lim_{n\to\infty}
    \frac{\ell_R(([(I+N)/N]^n)^{\sat}/[(I+N)/N]^n)}{n^d}
    \\&=\lim_{n\to\infty}
    \frac{\ell_R([(I^n+N)^{\sat}/N]/[(I^n+N)/N])}{n^d}
    \\&=\lim_{n\to\infty}
    \frac{\ell_R([(I^n+N)^{\sat}]/[(I^n+N)])}{n^d}
    \end{split}
\end{equation}exists (the third equation holds by Remark \ref{rmk2.1}). On the other hand, we have an exact sequence,
\begin{equation}\label{eq2}
    0\to \frac{(I^n)^{\sat}+N}{I^n+N}
    \to \frac{(I^n+N)^{\sat}}{I^n+N}
    \to \frac{(I^n+N)^{\sat}}{(I^n)^{\sat}+N}\to 0.
\end{equation} 

We will proceed by proving the following statement.

\begin{Proposition}\label{prop1}
    The limit
    \begin{equation*}
        \lim_{n\to\infty}\frac{\ell_R([(I^n+N)^{\sat}]/[(I^n)^{\sat}+N])}{n^d}
    \end{equation*}exists.
\end{Proposition}

In order to prove Proposition \ref{prop1} we will first prove several lemmas (Lemma \ref{lem2}, Remark \ref{rmk3}, Lemma \ref{lem3}, and Lemma \ref{lem4}). Let $m$ be the maximal ideal of $R$ and let $\overline{R} = R/N$. In addition, let $\Tilde{I} = I(R/N)$, which equals $(I+N)/N$.

\begin{Lemma}\label{lem2}
    There exists $c\in\mathbb{Z}_{>0}$ such that
    \begin{equation*}
        ((I^n+N)^{\sat})\cap m^{nc}
        = ((I^n)^{\sat}+N)\cap m^{nc}
    \end{equation*}for all $n\geq 0$.
\end{Lemma}
\begin{proof}
    By Theorem 3.4 \cite{S1}, there exists $c>0$ such that
    \begin{equation}\label{eq3}
        ((\Tilde{I}^n)^{\sat})\cap m_{\overline{R}}^{cn}
        = (\Tilde{I}^n)\cap m_{\overline{R}}^{cn}
    \end{equation}for all $n$. We have the following equations by Remark \ref{rmk2.1}
    
        \begin{equation}\label{eq4.1}
            m_{\overline{R}}^{cn}
            = ((m+N)/N)^{cn}
            =(m^{cn}+N)/N
        \end{equation}

        \begin{equation}\label{eq4.2}
            \Tilde{I}^n
            = \bigg(\frac{I+N}{N}\bigg)^n
            =\frac{I^n+N}{N}
        \end{equation}

        \begin{equation}\label{eq4.3}
            \big(\big(\Tilde{I}^n\big)^{\sat}\big)
            = \bigg(\bigg(\frac{I+N}{N} \bigg)^n\bigg)^{\sat}
            = \frac{(I^n+N)^{\sat}}{N}.
        \end{equation}

        By equations (\ref{eq3}), (\ref{eq4.1}), (\ref{eq4.2}), and (\ref{eq4.3}) we have
        \begin{equation}\label{eq4.4}
            \frac{(I^n+N)^{\sat}}{N}\cap \frac{m^{cn}+N}{N}
            =\frac{I^n+N}{N} \cap \frac{m^{cn}+N}{N}
        \end{equation}for all $n$.
        By equation (\ref{eq4.4}) and Remark \ref{rmk2.1} (iii) we have
        \begin{equation*}
            \frac{(I^n+N)^{\sat}\cap (m^{cn}+N)}{N}
            =\frac{(I^n+N)\cap (m^{cn}+N)}{N}.
        \end{equation*} for all $n$. By Lemma \ref{lem2-1},
        \begin{equation*}
            \begin{split}
                (I^n+N)^{\sat}\cap (m^{cn}+N)
            &= (I^n+N)\cap (m^{cn}+N)\\
            &= [(I^n+N)\cap m^{cn}]
            + [(I^n+N)\cap N]
            \\&=[(I^n+N)\cap m^{cn}] + N
            \end{split}
        \end{equation*} for all $n$. Consequently
        \begin{equation*}
        \begin{split}
            (I^n+N)^{\sat} \cap m^{cn}
            &\subset (I^n+N)^{\sat}\cap (m^{cn}+N)\\
            &= [(I^n+N)\cap m^{cn}] + N
            \\& \subset I^n + N \subset (I^n)^{\sat} + N
        \end{split}
        \end{equation*} so that
        \begin{equation}\label{eq4.5}
            (I^n+N)^{\sat} \cap m^{cn}
            \subset [(I^n)^{\sat} + N]\cap m^{cn}
        \end{equation} for all $n$. On the other hand, $(I^n)^{\sat} + N\subset (I^n+N)^{\sat}$ implies that
        \begin{equation}\label{eq 4.6}
            [(I^n)^{\sat} + N]\cap m^{cn}
            \subset (I^n+N)^{\sat} \cap m^{cn}
        \end{equation} for all $n$. Equations (\ref{eq4.5}) and (\ref{eq 4.6}) together yield
        \begin{equation*}
            [(I^n)^{\sat} + N]\cap m^{cn}
            = (I^n+N)^{\sat} \cap m^{cn}
        \end{equation*}for all $n$, which finishes the proof of the lemma.
\end{proof}

Now fix a positive number $c$ as in Lemma \ref{lem2}. We will define families of ideals $I_n$ and $J_n$ in $\overline{R}$ and show that $I_n\cap m_{\overline{R}}^{cn} = J_n\cap m_{\overline{R}}^{cn}$ for all $n$. From there we will prove Proposition \ref{prop1}.\\

Let
\begin{equation}
    \begin{split}
        J'_n = (I^n+N)^{\sat}&, I'_n = (I^n)^{\sat}+N,\\
        J_n = J'_n/N, &\text{ and } I_n = I'_n/N
    \end{split}
\end{equation}for $n\geq 0$. We have that $J_0=R=I_0$ (since $R^{\sat}=R$). Lemma \ref{lem2} states that $J'_n\cap m^{nc}=I'_n\cap m^{nc}$ for all $n$. Now we will prove the following remark.

\begin{Remark}\label{rmk3}
    $J_n\cap m_{\overline{R}}^{nc}= I_n\cap m_{\overline{R}}^{nc}$ for all $n$
\end{Remark}

\begin{proof}
Since $J'_n,\textbf{ } I'_n\supset N$, we have the following equations for all $n$
\begin{equation}\label{eq4.7}
    \begin{split}
        J'_n\cap (m^{nc}+N)
        &= (J'_n\cap m^{nc}) + (J'_n\cap N)
        = (J'_n\cap m^{nc})+N\\
        =(I'_n\cap m^{nc}) + N
        &= (I'_n\cap m^{nc}) + (I'_n\cap N)
        = I'_n\cap (m^{nc} + N).
    \end{split}
\end{equation}
Equation (\ref{eq4.7}) together with Remark \ref{rmk2.1} (ii) and (iii) yield that
\begin{equation*}
    \begin{split}
        \frac{J'_n\cap (m^{nc}+N)}{N}
        &= \frac{I'_n\cap (m^{nc} + N)}{N}\\
        (J'_n/N) \cap ((m^{nc}+N)/N)
        &= (I'_n/N) \cap ((m^{nc}+N)/N)\\
        J_n \cap [((m+N)/N)^{nc}]
        &= I_n \cap [((m+N)/N)^{nc}]\\
        J_n\cap m_{\overline{R}}^{nc}
        &= I_n\cap m_{\overline{R}}^{nc}
    \end{split}
\end{equation*}for all $n$.
\end{proof}

\begin{Lemma}\label{lem3}
    $\{I_n\}$ and $\{J_n\}$ are filtrations in $\overline{R}$.
\end{Lemma}

\begin{proof}
    Since $\{(I^n)^{\sat}\}$ is a filtration, we have that $\{I_n = (I^n)^{\sat}\overline{R}\}$ is a filtration.\\

    Now we will prove that $\{J_n\}$ is a filtration. Recall that $J_n = J'_n/N$ for each $n$. To see that $\{J_n\}$ is a filtration, it is enough prove that $\{J'_n = (I^n+N)^{\sat}\}$ is a graded family (the rest of the argument is similar to the above paragraph). Fix $i,j\geq 0$, $a\in (I^i+N)^{\sat}$, and $b\in (I^j+N)^{\sat}$. Let $t\in\mathbb{Z}_{>0}$ be so large that the equations
    \begin{equation*}
        \begin{split}
            (I^{i+j}+N)^{\sat}&= (I^{i+j}+N):m^{t'}\\
            (I^{i}+N)^{\sat}&= (I^{i}+N):m^{t}\\
            (I^{j}+N)^{\sat}&= (I^{j}+N):m^{t}\\
        \end{split}
    \end{equation*}hold for $t'\geq t$. We have
    \begin{equation}\label{eq5.1}
        m^{2t} = (y_1y_2\mid y_1,y_2\in m^t).
    \end{equation} For $y_1,y_2\in m^t$, we have
    \begin{equation*}
        \begin{split}
            aby_1y_2
            = (ay_1)(by_2)
            \in (I^i+N)(I^j+N)\subset (I^{i+j}+N).
        \end{split}
    \end{equation*}Now by (\ref{eq5.1}) we see that $ab\in (I^{i+j}+N):m^{2t} = (I^{i+j}+N)^{\sat}$, so that $J'_iJ'_j\subset J'_{i+j}$. Thus $\{J'_n\}$ is a graded family. Therefore, $\{J_n=J'_n/N\}$ is a graded family, and hence a filtration. This finishes the proof of the lemma.
\end{proof}

\begin{Lemma}\label{lem4}
    $\dim\overline{R} = d$ since every prime ideal of $R$ contains $N$.
\end{Lemma}

Now we are ready to prove Proposition \ref{prop1}.

\begin{proof}
    By Lemma \ref{lem4}, $\overline{R}$ is an analytically unramified ring of dimension $d>0$. Now by Remark \ref{rmk3}, Lemma \ref{lem3}, and Theorem 6.1 \cite{C2}, the limit
    \begin{equation*}
        \lim_{n\to\infty}\frac{\ell_{\overline{R}}(J_n/I_n)}{n^d}
        =\lim_{n\to\infty}\frac{\ell_R((I^n+N)^{\sat}/((I^n)^{\sat})+N)}{n^d}
    \end{equation*}exists. So Proposition \ref{prop1} is proven.
\end{proof}

Now we will prove Theorem \ref{thm1}.

\begin{proof}
    By the exact sequence (\ref{eq2}) we have
\begin{equation*}
    \ell_R([(I^n)^{\sat}+N]/[I^n+N])
    = \ell_R([(I^n+N)^{\sat}]/[I^n+N])
    -\ell_R([(I^n+N)^{\sat}]/[(I^n)^{\sat}+N])
\end{equation*}for all $n$. Theorem \ref{thm1} now follows from the existence of the limit in Equation (\ref{eq1}), Proposition \ref{prop1}, and Remark \ref{rmk2}.
\end{proof}

\begin{Lemma}\label{completionLemma}
    Suppose that $R$ is a Noetherian local ring with maximal ideal $m$, and $M$ is a finitely generated Artin $R$-module. Let $\widehat{M}$ be the $m$-adic completion of $M$. Then the map
    \begin{equation*}
        \begin{split}
            M&\to \widehat{M} = \Big\{(a_1+mM,\ldots )\in \prod_{n\geq 1}M/m^nM\mid a_j-a_i\in m^iM \text{ for $j>i\geq 1$}\Big\}\\
            a&\mapsto (a+mM,a+m^2M,\ldots)
        \end{split}
    \end{equation*} is an isomorphism of $R$-modules.
\end{Lemma}
\begin{proof}Let $t\in\mathbb{Z}_{>0}$ be large enough that $m^tM = m^iM$ for $i\geq t$. By Krull's intersection theorem, $\cap_{n\geq 1}m^nM=0$, so that $0 = m^tM = m^iM$ for $i\geq t$. Let\\ $\phi:M\to  \widehat{M}$ be the $R$-module homomorphism defined by\\ $a\mapsto (a+mM,a+m^2M,\ldots)$ for $a\in M$. If $a\in M$ and $\phi(a) = 0$, then\\ $a\in \cap_{n\geq 0}m^nM=0$ so that $a = 0$. Thus $\phi$ is injective. Suppose that\\ $\alpha = (a_i+m^iM)_{i\geq 1}\in \widehat{M}$. Then $a_i-a_t\in m^tM=0$ for $i\geq t$ so that $a_i=a_t$ for $i\geq t$. Further, $a_i - a_t\in m^iM$ for $i\leq t$ so that $\phi(a_t) = \alpha$. Thus $\phi$ is surjective and so $\phi$ is an isomorphism.
\end{proof}

Now we prove Corollary \ref{thm2}, which is the case of Theorem \ref{TheoremA} when $d>0$.

\begin{Corollary}\label{thm2}
    Let $R$ be a Noetherian local ring of dimension $d>0$ such that the nildradical of $\widehat{R}$ has dimension less than $d$. Let $I\subset R$ be any ideal. Then the limit
    \begin{equation*}
        \lim_{n\to\infty}\frac{\ell_R((I^n)^{\sat}/I^n)}{n^d}
    \end{equation*}exists.
\end{Corollary}
\begin{proof}
There exists $t\in\mathbb{Z}_{>0}$ such that
\begin{equation*}
    (I^n)^{\sat} = I^n:m_R^t\text{ and }(I^n\widehat{R})^{\sat}
    = I^n\widehat{R}:m_{\widehat{R}}^t
\end{equation*}since $R$ and $\widehat{R}$ are Noetherian. By flatness of $R\to\widehat{R}$ we now have
\begin{equation}\label{eq5.2}
    ([I\widehat{R}]^n)^{\sat}
     = (I^n\widehat{R})^{\sat} 
     = I^n\widehat{R}: m_R^t\widehat{R}
     = (I^n:m_R^t)\widehat{R}
     = (I^n)^{\sat}\widehat{R}.
\end{equation}

We will now show that the $R$-modules $M:= (I^n)^{\sat}/I^n$ and $[(I\widehat{R})^n]^{\sat}/(I\widehat{R})^n$ are isomorphic. Fix $n\in\mathbb{N}$. There exists $C\in \mathbb{Z}_{>0}$ such that $(I^n)^{\sat}\cap m_R^C = I^n\cap m_R^C$, and so $m_R^C((I^n)^{\sat}/I^n)=0$. Then $M$ is an $R/m_R^C$-module, and is hence a finitely generated Artin $R$-module. By Lemma \ref{completionLemma} the map $M\to \widehat{M}$ is an $R$-module isomorphism, where $\widehat{M}$ is the $m_R$-adic completion of $M$. This combined with (\ref{eq5.2}) gives that we have an $R$-module isomorphism
\begin{equation}\label{eq5.4}
    (I^n)^{\sat}/I^n
    \cong \widehat{(I^n)^{\sat}/I^n}
    = [(I^n)^{\sat}\widehat{R}]/(I^n\widehat{R})
    = [(I\widehat{R})^n]^{\sat}/(I\widehat{R})^n.
\end{equation}

Since $m_R^C((I^n)^{\sat}/I^n)=0$,
\begin{equation*}
m_R^C[((I^n)^{\sat}\widehat{R})/(I^n\widehat{R})]=0
\end{equation*}
by the $R$-module isomorphism (\ref{eq5.4}).
Thus, $[(I\widehat{R})^n]^{\sat}/(I\widehat{R})^n$ is an $R/m_R^C\cong \widehat{R}/m_R^C\widehat{R}$-module.\\ Therefore
\begin{equation*}
    \begin{split}
        \ell_R((I^n)^{\sat}/I^n)
    &= \ell_R([(I\widehat{R})^n]^{\sat}/(I\widehat{R})^n)
    = \ell_{R/m_R^C}([(I\widehat{R})^n]^{\sat}/(I\widehat{R})^n)
    \\&=\ell_{\widehat{R}/m_{R}^C\widehat{R}}([(I\widehat{R})^n]^{\sat}/(I\widehat{R})^n)
    = \ell_{\widehat{R}}([(I\widehat{R})^n]^{\sat}/(I\widehat{R})^n)
    \end{split}
\end{equation*}for all $n\in\mathbb{N}$. The conclusion now follows by applying Theorem \ref{thm1} to the ideal $I\widehat{R}$ of the complete local ring $\widehat{R}$.
\end{proof}

\begin{Theorem}\label{thm3}
    Let $R$ be any $d$-dimensional Noetherian local ring such that the nildradical of $\widehat{R}$ has dimension less than $d$. Let $I\subset R$ an ideal. Then the limit
    \begin{equation*}
        \lim_{n\to\infty}\frac{\ell_R((I^n)^{\sat}/I^n)}{n^d}
    \end{equation*}exists.
\end{Theorem}

\begin{proof}
    By Corollary \ref{thm2}, we may assume that $d:=\dim(R)=0$. Since $R$ is Artinian, there is a positive number $t$ so that $I^n = I^{t}$ for $n\geq t$. Then if $A:= \ell_R((I^t)^{\sat}/I^t)$, $\ell_R((I^n)^{\sat}/I^n)=A$ for $n\geq t$ so that the limit
    \begin{equation*}
        \lim_{n\to\infty}\frac{\ell_R((I^n)^{\sat}/I^n)}{n^d}
        = \lim_{n\to\infty}A/n^0=A
    \end{equation*}exists.
\end{proof}

Now we begin proving Theorem \ref{TheoremB}. We first recall the definition of the Amao multiplicity.

\begin{Definition}
    Let $R$ be a Noetherian local ring of dimension $d$ and let $I\subset J$ be ideals of $R$ such that $\ell_R(J/I)<\infty$. Then $\ell_R(J^n/I^n)<\infty$ for all $n>1$ and the limit
    \begin{equation*}
        a(I,J):=\lim_{n\to\infty}\frac{\ell_R(J^n/I^n)}{n^d/d!}
    \end{equation*}exists. $a(I,J)$ is called the Amao multiplicity. Amao multiplicities are developed in \cite{A1}, \cite{R1}, and page 332 of \cite{SH}.
\end{Definition}

We will prove the following statement.

\begin{Proposition}\label{thm4}
    Let $R$ be a $d$-dimensional complete (Noetherian) local ring, let $I\subset R$ be any ideal, and let $N$ be the nildradical of $R$. Suppose that $\dim(N)<d$. Then the limit
    \begin{equation*}
        \lim_{m\to\infty}\frac{a(I^m,(I^m)^{\sat})}{m^d}
    \end{equation*} exists, and
\begin{equation}\label{eq6.0}
    \varepsilon(I)=\lim_{m\to\infty}\frac{a(I^m,(I^m)^{\sat})}{m^d}.
\end{equation}
\end{Proposition}

If $N = 0$, then $R$ is analytically unramified. Consequently, the limits in (\ref{eq6.0}) exist and are equal by \cite{CL} Theorem 1.1.\\

Let $\overline{R} = R/N$, and let $\Tilde{I}:= I\overline{R}$, which equals $(I+N)/N$. We prove Proposition \ref{thm4} by first proving several lemmas (Lemma \ref{lem5}, Lemma \ref{lem6}, Proposition \ref{prop2}, and Lemma \ref{lem7}).

\begin{Lemma}\label{lem5}
    The limit
    \begin{equation*}
        \lim_{n\to\infty}\frac{\ell_R((I^n+N)^{\sat}/(I^n+N))}{n^d}
    \end{equation*}exists, the limits
    \begin{equation*}
        \lim_{k\to\infty}\frac{\ell_R(([(I^m+N)^{\sat}]^k+N)/(I^{mk}+N))}{k^d}
    \end{equation*}exist, the limit
    \begin{equation*}
        \lim_{m\to\infty}\frac{1}{m^d}\lim_{k\to\infty}\frac{\ell_R(([(I^m+N)^{\sat}]^k+N)/(I^{mk}+N))}{k^d}
    \end{equation*}exists, and we have the following formula
    \begin{equation}\label{eq6.1}
        \begin{split}
            \lim_{n\to\infty}&\frac{\ell_R((I^n+N)^{\sat}/(I^n+N))}{n^d}
        \\&=\lim_{m\to\infty}\frac{1}{m^d}\lim_{k\to\infty}\frac{\ell_R(([(I^m+N)^{\sat}]^k+N)/(I^{mk}+N))}{k^d}.
        \end{split}
    \end{equation}
\end{Lemma}
\begin{proof}
    By Remark \ref{rmk2.1} (i) and (ii), we have $R$-module isomorphisms
    \begin{equation}\label{eq6.2'}
        \frac{(I^n+N)^{\sat}}{I^n+N}
        \cong \frac{(I^n+N)^{\sat}/N}{(I^n+N)/N}
        = \frac{(\Tilde{I}^n)^{\sat}}{\Tilde{I}^n}
    \end{equation}
    \begin{equation}\label{eq6.2''}
        \frac{[(I^m+N)^{\sat}]^k+N}{I^{mk}+N}
        \cong \frac{([(I^m+N)^{\sat}]^k+N)/N}{(I^{mk}+N)/N}
        = \frac{[(\Tilde{I}^m)^{\sat}]^k}{\Tilde{I}^{mk}}
    \end{equation}
    for all $n$, $m$, and $k$. Since $\dim(\overline{R})=d$, (\ref{eq6.2'}) implies that
    \begin{equation}\label{eq6.3'}
        \frac{1}{d!}\varepsilon(\Tilde{I}) = \lim_{n\to\infty}\frac{\ell_R((I^n+N)^{\sat}/(I^n+N))}{n^d}.
    \end{equation} (\ref{eq6.2''}) implies that
    \begin{equation}\label{eq6.3''}
        \frac{1}{d!}a(\Tilde{I}^m,(\Tilde{I}^m)^{\sat})
        = \lim_{k\to\infty}\frac{\ell_R(([(I^m+N)^{\sat}]^k+N)/(I^{mk}+N))}{k^d}
    \end{equation}for all $m$. Theorem 1.1 \cite{CL} applied to the ideal $\Tilde{I}$ in the $d$-dimensional complete reduced ring $\overline{R}$ yields that
    \begin{equation}\label{eq6.4}
        \frac{1}{d!}\varepsilon(\Tilde{I}) = \lim_{m\to\infty}\frac{1}{m^d}\frac{1}{d!}a(\Tilde{I}^m,(\Tilde{I}^m)^{\sat}).
    \end{equation}The proof of the lemma is complete by equations (\ref{eq6.3'}), (\ref{eq6.3''}), and (\ref{eq6.4}).
\end{proof}

Now we will define filtrations in $\overline{R}$ to which Theorem 4.1 \cite{CL} applies. We have ideals in $\overline{R}$
\begin{equation}\label{eqfilt}
    \begin{split}
        I_n&= ((I^n)^{\sat}+N)/N\\
        J_n&= (I^n+N)^{\sat}/N\\
        I(m)_k &:= I_m^k = ([(I^m)^{\sat}]^k+N)/N\\
        J(m)_k &:= J_m^k = ([(I^m+N)^{\sat}]^k+N)/N
    \end{split}
\end{equation}for all $n$, $m$, and $k$ (the preceding two lines follow from Remark \ref{rmk2.1} (ii)). The families of ideals $\cal{I}:=\{I_n\}$, $\cal{J}:=\{J_n\}$ $\cal{I}(m) :=\{I(m)_k\}$, and $\cal{J}(m):= \{J(m)_k\}$ are filtrations in $\overline{R}$ for all $m$ by Lemma \ref{lem3}.

\begin{Lemma}\label{lem6}
    There exists a positive number $C$ such that
    \begin{equation*}
        J(m)_k\cap m_{\overline{R}}^{Cmk}
        =I(m)_k\cap m_{\overline{R}}^{Cmk}
    \end{equation*}for all $m$ and $k$.
\end{Lemma}
\begin{proof}
    By Lemma 5.2 \cite{CL} applied to the ideal $\Tilde{I}$ in $\overline{R}$, there exists a positive number $C$ such that
    \begin{equation*}
        \Tilde{I}^{mk}\cap m_{\overline{R}}^{Cmk}
        = [(\Tilde{I}^m)^{\sat}]^k\cap m_{\overline{R}}^{Cmk}
    \end{equation*}for all $m$ and $k$. This combined with Remark \ref{rmk2.1} (i) and (ii) yields that
    \begin{equation*}
        \frac{I^{mk}+N}{N}\cap m_{\overline{R}}^{Cmk}
        = \Bigg(\frac{(I^m+N)^{\sat}+N}{N}\Bigg)^k\cap m_{\overline{R}}^{Cmk}
        = \frac{[(I^m+N)^{\sat}]^k+N}{N}\cap m_{\overline{R}}^{Cmk}
    \end{equation*} for all $m$ and $k$. Thus
    \begin{equation*}
        \begin{split}
            J(m)_k\cap m_{\overline{R}}^{Cmk}
            \subset
            \frac{I^{mk}+N}{N}\cap m_{\overline{R}}^{Cmk}
            \subset \frac{[(I^m)^{\sat}]^k+N}{N}\cap m_{\overline{R}}^{Cmk}
            = I(m)_k\cap m_{\overline{R}}^{Cmk}
        \end{split}
    \end{equation*}so that $J(m)_k\cap m_{\overline{R}}^{Cmk}
        =I(m)_k\cap m_{\overline{R}}^{Cmk}$ for all $m$ and $k$.
\end{proof}

\begin{Proposition}\label{prop2}
The limit
\begin{equation*}
    \lim_{m\to\infty}\frac{1}{m^d}\lim_{k\to\infty}
        \frac{\ell_R([((I^m+N)^{\sat})^k+N]/[((I^m)^{\sat})^k+N])}{k^d}
\end{equation*}exists and we have the following formula
    \begin{equation}\label{eq7}
        \begin{split}
            \lim_{n\to\infty}&\frac{\ell_R((I^n+N)^{\sat}/((I^n)^{\sat}+N))}{n^d}
        \\&=\lim_{m\to\infty}\frac{1}{m^d}\lim_{k\to\infty}
        \frac{\ell_R([((I^m+N)^{\sat})^k+N]/[((I^m)^{\sat})^k+N])}{k^d}.
        \end{split}
    \end{equation}
\end{Proposition}
\begin{proof}
We will first prove the proposition when $d>0$. Assume that $d>0$. Since $\{J_n\}$ and $\{I_n\}$ are graded families, we have that $J(m)_1 = J_m$, $I(m)_1 = I_m$, $I(m)_k = I_m^k\subset I_{mk}$, $J(m)_k = J_m^k\subset J_{mk}$, and $I(m)_k\subset J(m)_k$ for all $m$ and $k$. Now by Lemma \ref{lem6} and Theorem 4.1 \cite{CL}, we have that
\begin{equation*}
    \lim_{n\to\infty}\frac{\ell_R(J_n/I_n)}{n^d}
    =
    \lim_{m\to\infty}\frac{1}{m^d}
    \lim_{k\to\infty}
    \frac{\ell_R(J(m)_k/I(m)_k)}{k^d}.
\end{equation*}
Equation (\ref{eq7}) now follows from (\ref{eqfilt}).\\

Now we must prove the proposition when $d = 0$. First suppose that $I = R$. Then for all $n$, $m$, and $k$, we have $I_n = R/N = J_n$, and $I(m)_k = I_m^k = (R/N)^k = R/N$, and similarly $J(m)_k = R/N$. Consequently,
\begin{equation*}
    \lim_{n\to\infty}\frac{\ell_R(J_n/I_n)}{n^d}
    = 0
    = \lim_{m\to\infty}\frac{1}{m^d}
    \lim_{k\to\infty}
    \frac{\ell_R(J(m)_k/I(m)_k)}{k^d}.
\end{equation*}Thus we may assume that $I\subset m_R$.\\

$R$ is Artinian, so we have that $I^n, N^k, m_R^t = 0$ for $n$, $k$, and $t$ sufficiently large. Thus for any ideal $J$ of $R$, $J^{\sat} = J:m_R^{\infty} = J:0 = R$. Now we have that
\begin{equation*}
    \begin{split}
        I_n &= R/N\\
        J_n &= N^{\sat}/N=R/N\\
        I(m)_k &= I_m^k= (R/N)^k = R/N\\
        J(m)_k &= J_m^k = (R/N)^k = R/N
    \end{split}
\end{equation*}for $n$, $m$, and $k$ large enough. Consequently,
\begin{equation*}
    \lim_{n\to\infty}\frac{\ell_R(J_n/I_n)}{n^d}
    = 0 = 
    \lim_{m\to\infty}\frac{1}{m^d}
    \lim_{k\to\infty}
    \frac{\ell_R(J(m)_k/I(m)_k)}{k^d}.
\end{equation*} which finishes the proof of the proposition.
\end{proof}

\begin{Lemma}\label{lem7}
    For all $m$, we have that
    \begin{equation*}
        \frac{1}{d!}a(I^m,(I^m)^{\sat})
        = \lim_{k\to\infty}
        \frac{\ell_R(([(I^m)^{\sat}]^k+N)/(I^{mk}+N))}{k^d}.
    \end{equation*}
\end{Lemma}
\begin{proof}
    We have an exact sequence
    \begin{equation}\label{eq8}
        0\to\frac{[(I^m)^{\sat}]^k\cap (I^{mk}+N)}{I^{mk}}\to \frac{[(I^m)^{\sat}]^k}{I^{mk}}\to \frac{[(I^m)^{\sat}]^k+N}{I^{mk}+N}\to 0
    \end{equation} while Lemma \ref{lem2-1} implies that
    \begin{equation*}
        \frac{[(I^m)^{\sat}]^k\cap (I^{mk}+N)}{I^{mk}}= \frac{I^{mk}+([(I^m)^{\sat}]^k\cap N)}{I^{mk}}\cong \frac{[(I^m)^{\sat}]^k\cap N}{I^{mk}\cap N}.
    \end{equation*} By (\ref{eq8}) we now have an exact sequence
    \begin{equation}\label{eq8.1}
        0\to\frac{[(I^m)^{\sat}]^k\cap N}{I^{mk}\cap N}\to \frac{[(I^m)^{\sat}]^k}{I^{mk}}\to \frac{[(I^m)^{\sat}]^k+N}{I^{mk}+N}\to 0
    \end{equation}By Lemma 5.2 \cite{CL} there exists $c_1>0$ such that
    \begin{equation}\label{eq8.2}
        [(I^m)^{\sat}]^k\cap m_R^{c_1 mk} = I^{mk}\cap m_R^{c_1 mk}
    \end{equation}for all $m$ and $k$. For all $m$ and $k$, we have a natural surjection
    \begin{equation*}
        \frac{[(I^m)^{\sat}]^k \cap N}{I^{mk}\cap Nm_R^{c_1 mk}}\to \frac{([(I^m)^{\sat}]^k \cap N)/(I^{mk}\cap Nm_R^{c_1 mk})}{(I^{mk}\cap N)/(I^{mk}\cap Nm_R^{c_1 mk})} \cong \frac{[(I^m)^{\sat}]^k\cap N}{I^{mk}\cap N}
    \end{equation*} and a natural map
    \begin{equation*}
        \frac{[(I^m)^{\sat}]^k \cap N}{I^{mk}\cap Nm_R^{c_1 mk}}\to \frac{N}{Nm_R^{c_1 mk}}
    \end{equation*}which is injective by (\ref{eq8.2}). Thus,
    \begin{equation*}
        \ell_R\bigg(\frac{[(I^m)^{\sat}]^k\cap N}{I^{mk}\cap N}\bigg) \leq \ell_R(N/Nm_R^{c_1mk}).
    \end{equation*}This combined with the exact sequence (\ref{eq8.1}) completes the proof of the lemma, since $\dim N<d$.
\end{proof}

We are now ready to prove Proposition \ref{thm4}. Note that $((I^m)^{\sat})^k+N\subset ((I^m+N)^{\sat})^k+N$ for all $m$ and $k$. We have an exact sequence
\begin{equation}
    0\to \frac{[(I^m)^{\sat}]^k+N}{I^{mk}+N}
    \to \frac{[(I^m+N)^{\sat}]^k+N}{I^{mk}+N}
    \to \frac{[(I^m+N)^{\sat}]^k+N}{[(I^m)^{\sat}]^k+N}\to 0
\end{equation} so that
\begin{equation}\label{eq9.1}
    \begin{split}
        \ell_R(([(I^m)^{\sat}]^k+N)/&(I^{mk}+N))
    = \ell_R(([(I^m+N)^{\sat}]^k+N)/(I^{mk}+N))
    \\&- \ell_R(([(I^m+N)^{\sat}]^k+N)/([(I^m)^{\sat}]^k+N))
    \end{split}
\end{equation}for all $m$ and $k$. The exact sequence (\ref{eq2}) implies that
\begin{equation}\label{eq9.2}
    \begin{split}
        \ell_R(((I^n)^{\sat}+N)/(I^n+N))
    &= \ell_R((I^n+N)^{\sat}/(I^n+N))
    \\&- \ell_R((I^n+N)^{\sat}/((I^n)^{\sat}+N))
    \end{split}
\end{equation}for all $n$. Subtracting (\ref{eq7}) from (\ref{eq6.1}) yields that
\begin{equation}\label{eq9.3}
     \begin{split}
            &\lim_{n\to\infty}\frac{\ell_R((I^n+N)^{\sat}/(I^n+N))-\ell_R((I^n+N)^{\sat}/((I^n)^{\sat}+N))}{n^d}
        \\&=\lim_{m\to\infty}\frac{1}{m^d}\lim_{k\to\infty}\frac{\ell_R\Big(\frac{[(I^m+N)^{\sat}]^k+N}{I^{mk}+N}\Big)-\ell_R\Big(\frac{[(I^m+N)^{\sat}]^k+N}{[(I^m)^{\sat}]^k+N}\Big)}{k^d}
        \end{split}
\end{equation} Rewriting (\ref{eq9.3}) using (\ref{eq9.1}) and (\ref{eq9.2}) yields
\begin{equation*}
    \lim_{n\to\infty}\frac{\ell_R(((I^n)^{\sat}+N)/(I^n+N))}{n^d}
    =\lim_{m\to\infty}\frac{1}{m^d}
    \lim_{k\to\infty}\frac{\ell_R(([(I^m)^{\sat}]^k+N)/(I^{mk}+N))}{k^d}.
\end{equation*}Now by Lemma \ref{lem7} we have
\begin{equation}\label{eq9.4}
    \lim_{n\to\infty}\frac{\ell_R(((I^n)^{\sat}+N)/(I^n+N))}{n^d}
    =\lim_{m\to\infty}\frac{1}{m^d}
    \frac{1}{d!}a(I^m,(I^m)^{\sat}).
\end{equation}
Remark \ref{rmk2} is valid regardless of whether $d>0$ or $d=0$. We have that\\ $\lim_{n\to \infty}\ell_R(((I^n)^{\sat}+N)/(I^n+N))n^{-d}$ exists by (\ref{eq9.2}), Proposition \ref{prop2}, and Lemma \ref{lem5}. Now Remark \ref{rmk2} and equation (\ref{eq9.4}) together imply that
\begin{equation*}
    \lim_{n\to\infty}\frac{\ell_R((I^n)^{\sat}/I^n)}{n^d}
    =\frac{1}{d!} \lim_{m\to\infty}\frac{a(I^m,(I^m)^{\sat})}{m^d}
\end{equation*} so that
\begin{equation*}
    \varepsilon(I) = \lim_{m\to\infty}\frac{a(I^m,(I^m)^{\sat})}{m^d}
\end{equation*} which finishes the proof of Proposition \ref{thm4}.\\

Proposition \ref{thm4} has the following consequence.

\begin{Theorem}\label{thm5}
    Let $R$ be a $d$-dimensional Noetherian local ring and let $I\subset R$ be an ideal in $R$. Suppose that the dimension of the nildradical of $\widehat{R}$ is less than $d$. Then the limit
    \begin{equation*}
        \lim_{m\to\infty}\frac{a(I^m,(I^m)^{\sat})}{m^d}
    \end{equation*} exists, and
\begin{equation*}
    \varepsilon(I)=\lim_{m\to\infty}\frac{a(I^m,(I^m)^{\sat})}{m^d}.
\end{equation*}
\end{Theorem}

\begin{proof}
    The fact from the proof of Corollary \ref{thm2} that
    \begin{equation*}
        \ell_R((I^n)^{\sat}/I^n)
        =\ell_{\widehat{R}}\big(\big[\big(I\widehat{R}\big)^n\big]^{\sat}/\big(I\widehat{R}\big)^n\big)
    \end{equation*} for $n\geq 1$ is valid even when $d=0$. Consequently,
    \begin{equation*}
        \varepsilon(I)
        = \varepsilon(I\widehat{R}).
    \end{equation*}This combined with Proposition \ref{thm4} yields that
    \begin{equation}\label{eq10.1}
        \varepsilon(I)
        = \lim_{m\to\infty}
        \frac{a((I\widehat{R})^m,[(I\widehat{R})^m]^{\sat})}{m^d}.
    \end{equation} Fix $m,k\geq 1$ and let $M = [(I^m)^{\sat}]^k/I^{mk}$. Next we will show that
    \begin{equation*}
        \ell_R(M)
        =\ell_{\widehat{R}}(([(I\widehat{R})^m]^{\sat})^k/(I\widehat{R})^{mk}).
    \end{equation*}By Lemma 5.2 \cite{CL}, there exists $D\in\mathbb{Z}_{>0}$ such that $m_R^DM=m_R^D([(I^m)^{\sat}]^k/I^{mk})=0$. So $M$ is an $R/m_R^D$-module. Consequently $M$ is a finitely generated Artin $R$-module, so Lemma \ref{completionLemma} implies that the map
    \begin{equation*}
        [(I^m)^{\sat}]^k/I^{mk}
        \to ([(I\widehat{R})^m]^{\sat})^k/(I\widehat{R})^{mk}
    \end{equation*}
    is an $R$-module isomorphism.
    Now since $m_R^D([(I^m)^{\sat}]^k/I^{mk})=0$,\\ $m_R^D[([(I\widehat{R})^m]^{\sat})^k/(I\widehat{R})^{mk}]=0$ so that
    \begin{equation*}
        \begin{split}
            \ell_R([(I^m)^{\sat}]^k/I^{mk})
            &=
            \ell_R(([(I\widehat{R})^m]^{\sat})^k/(I\widehat{R})^{mk})
            =\ell_{R/m_R^D}(([(I\widehat{R})^m]^{\sat})^k/(I\widehat{R})^{mk})
            \\&=\ell_{\widehat{R}/m_{\widehat{R}}^D}(([(I\widehat{R})^m]^{\sat})^k/(I\widehat{R})^{mk})
            =\ell_{\widehat{R}}(([(I\widehat{R})^m]^{\sat})^k/(I\widehat{R})^{mk})
        \end{split}
    \end{equation*}for all $m$ and $k$. Therefore
    \begin{equation*}
        \begin{split}
            \frac{1}{d!}a((I\widehat{R})^m,[(I\widehat{R})^m]^{\sat})
        & = \lim_{k\to\infty}\frac{\ell_{\widehat{R}}\Big(([(I\widehat{R})^m]^{\sat})^k/(I\widehat{R})^{mk}\Big)}{k^d}
        \\&=
        \lim_{k\to\infty}\frac{\ell_R([(I^m)^{\sat}]^k/I^{mk})}{k^d}
        =\frac{1}{d!}
        a(I^m,(I^m)^{\sat}).
        \end{split}
    \end{equation*}
    This combined with (\ref{eq10.1}) completes the proof of the theorem.
\end{proof}

\end{document}